\documentclass{amsart}

\usepackage{cite}
\usepackage{amsmath,amsfonts}
\usepackage{amssymb}
\usepackage{tcolorbox}
\usepackage{amsthm}
\usepackage{graphicx}
\graphicspath{{./figs/}}
\usepackage{fixltx2e}
\usepackage{array}
\usepackage{mathtools}
\usepackage{comment}
\numberwithin{equation}{section}

\newtheorem{theorem}{Theorem}
\newtheorem{lemma}{Lemma}
\newtheorem{corollary}[lemma]{Corollary}

\newtheorem{remark}[lemma]{Remark}

\numberwithin{lemma}{section}

\begin{document}
\title{On the Existence of a Closed, Embedded, Rotational $\lambda$-Hypersurface}
\author{John Ross}
\address{Department of Mathematics and Computer Science, Southwestern University, 1001 E University Ave, Georgetown, TX 78626}
\email{rossjo@southwestern.edu}

\maketitle

\begin{abstract}
In this paper we show the existence of a closed, embedded $\lambda$-hypersurfaces $\Sigma \subset \mathbb{R}^{2n}$. The hypersurface $\Sigma$ is diffeomorhic to $\mathbb{S}^{n-1} \times \mathbb{S}^{n-1} \times \mathbb{S}^1$ and exhibits $SO(n) \times SO(n)$ symmetry. Our approach uses a ``shooting method'' similar to the approach used by McGrath in constructing a generalized self-shrinking ``torus'' solution to mean curvature flow. The result generalizes the $\lambda$ torus found by Cheng and Wei.
\end{abstract}

\section{Introduction}

In the study of mean curvature flow, an important class of solutions are those in which the hypersurface evolves under self-similar shrinking. Indeed, under general mean curvature flow, singularities often develop which can be modeled using self-shrinking solutions \cite{Huisken90}.  Such solutions can be identified with a single time-slice of the flow, which gives us a hypersurface called a self-shrinker. A self-shrinker satisfies the equation
\begin{align}
H = \frac{1}{2} \langle x,\nu \rangle
\end{align}
in which $H$ is the mean curvature of the hypersurface, $x$ is the position vector of the hypersurface, and $n$ is the vector normal to the hypersurface, with orientation chosen so that $\vec{H} = -H n$. Self-shrinkers are also notable because they are critical points of the weighted area functional 
\begin{align} \label{1.2}
F(\Sigma) = \int_\Sigma e^{-|x|^2 / 4}\; d\mu
\end{align}
and are minimal surfaces in the space $\mathbb{R}^{n+1}$ when imbued with the metric
\begin{align} \label{1.3}
e^{-\frac{|x|^2}{2(k+1)}}\sum_{i=1}^{k+1}(dx^i)^2.
\end{align}

A generaliztion of self-shrinkers leads to a class of hypersurfaces that are called $\lambda$-hypersurfaces. Such surfaces satisfy the  equation 
\begin{align} 
H = \frac{1}{2} \langle x,\nu \rangle + \lambda \label{1.4}
\end{align}
where $\lambda$ is a constant. These surfaces may be viewed as critical points to the weighted area functional \eqref{1.2} with respect to \textbf{weighted volume-preserving} variations -  that is, variations for which the function describing the normal direction of the variation, $u(x)$, satisfies $\int_\Sigma u e^{-|x|^2/4}\;d\mu = 0$. They may also be viewed as stationary solutions to the isoperimetric problem on the Gaussian space with metric given by \eqref{1.3}. More information on $\lambda$-hypersurfaces, including a derivation of this viewpoint, can be found in \cite{ChengWei14} and \cite{McGonagleRoss13}


There are very few explicit examples of complete embedded self-shrinkers, despite their importance in the field. The simplest (and most important) examples are generalized cylinders $\mathbb{R}^k \times \mathbb{S}^{n-k} \subset \mathbb{R}^{n+1}$ that are centered around the origin with radius $\sqrt{2(n-k)}$. In \cite{CM12}, Colding and Minicozzi showed that such self-shrinkers were \emph{generic}, in the sense that that under the mean-curvature flow on a generic hypersurface, the singularities that would develop look like these generalized cylinders. However, other special examples of self-shrinkers exist. In 1992, Angenent showed the existence of an embedded self-shrinker of genus 1, diffeomorphic to $\mathbb{S}^1 \times \mathbb{S}^{n-1}$ \cite{Angenent92}. In \cite{KKM12} and also \cite{Moller11}, self-shrinkers of arbitrary but large genus are constructed that contain a discrete rotational symmetry. Finally, in 2015, Peter McGrath \cite{McGrath15} constructed a closed self-shrinker that contains two rotational symmetries and is diffeomorphic to $\mathbb{S}^k \times \mathbb{S}^k \times \mathbb{S}^1$.

Broadly speaking, there are two established techniques for constructing new self-shrinkers. The first, practiced by Angenent and by McGrath, is to construct rotationally symmetric solutions by finding a ``generating curve'' in a lower-dimensional space. The second technique, employed by \cite{KKM12}, uses a gluing technique to adjoin preexisting self-shrinkers and minimal surfaces.

Even less has been done to construct examples of $\lambda$-hypersurfaces. In \cite{McGonagleRoss13}, it was shown that generalized cylinders were $\lambda$-hypersurfaces. In \cite{ChengWei15}, the authors construct the first nontrivial example of a $\lambda$-hypersurface, diffeomorphic to $\mathbb{S}^{n} \times \mathbb{S}^1$, using techniques similar to Angenent. There have also been non-trivial examples of one-dimensional self-shrinkers discovered in \cite{Chang14}. The aim of this paper is describe a new closed, embedded $\lambda$-hypersurface with a $\mathbb{S}^n \times \mathbb{S}^n$ symmetry. Our main result is:

\begin{theorem}
Let $n > 1$, and let $\lambda < 0$. Then there exists a $\lambda$-hypersurface $\Sigma^{2n + 1} \subset \mathbb{R}^{2n + 2}$ that is $\lambda$-hypersurface is diffeomorphic to $\mathbb{S}^n \times \mathbb{S}^n \times \mathbb{S}^1$ and exhibits a $O(n) \times O(n)$ rotational symmetry.
\end{theorem}

Although there are no a priori restrictions on $\lambda$ in a $\lambda$-hypersurface, we will assume throughout this paper that $\lambda < 0$. Note that requiring $\lambda < 0$ is a necessary technical requirement for our proof, see (for example) Lemma \ref{lemma32}. The condition is interesting, and a similar requirement has been necessary in  \cite{Chang14} and \cite{ChengWei15}. 

To prove this theorem, we employ a method similar to McGrath \cite{McGrath15}. We first determine a relationship between the $\lambda$-hypersurfaces we are interested in, and a ``generating curve'' in the first quadrant that satisfies a system of ODEs. We then construct a closed, embedded generating curve that satisfies this system of ODEs, by using a ``shooting method'' in the spirit of \cite{Angenent92}, \cite{McGrath15}. Of note is that, under the system of ODEs developed here, very few useful solutions are known to exist. In particular, the argument present in \cite{McGrath15} made use of the linear solution $y = x$, which we will not have access to.

The paper will be organized as follows: First, in Section 2, we will construct our system of ODEs to reduce the problem to finding a generating curve. This is similar to the treatment given in \cite{McGrath15}, but is included for completeness. In Section 3, we will analyze the system of ODEs to determine possible behavior of solutions. And finally, in Section 4, we employ our shooting method.

\section{Constructing the system of ODEs}

We are interested in studing $\lambda$-hypersurfaces that satisfy a rotational invariance under $O(m) \times O(n)$ for $m,n > 1$. To this end, let $O(m) \times O(n)$ act on $\mathbb{R}^{m+n} = \{ (\vec{x},\vec{y}): \vec{x} \in \mathbb{R}^m, \vec{y} \in \mathbb{R}^n \}$ in the usual way. We can then identify the space of orbits $\mathbb{R}^{m+n} / (O(m) \times O(n))$ with the first quadrant $\{(x,y) \in \mathbb{R}^2 : x,y \geq 0 \}$ under the projection 

$$\Pi(\vec{x}, \vec{y}) = (|\vec{x}|, |\vec{y}|) = (x,y) $$ 

Under this identification, each point $(x,y)$ in the first quadrant corresponds to the immersed submanifold $\mathbb{S}^{m-1}(x) \times \mathbb{S}^{n-1}(y) \subset \mathbb{R}^{m+n}$ (where $\mathbb{S}^{k}(x)$ is the $k$-dimensional sphere of radius $x$, embedded in $\mathbb{R}^{k+1}$ and centered at the origin).

Let $\Sigma \subset \mathbb{R}^{m+n}$ be an embedded $\lambda$-hypersurface. We say that $\Sigma$ is invariant under $O(m) \times O(n)$ if the action preserves $\Sigma$. If $\Sigma$ is invariant under $O(m) \times O(n)$, then the projection $\Pi (\Sigma)$ will give us a profile curve in the first quadrant, which we can parametrize by Euclidean arc length and write as $\gamma(t) = (x(t), y(t))$.

Recall that our $\lambda$-hypersurface satisfies the curvature equation \eqref{1.4}. Because $\Sigma$ is rotationally invariant, we can calculate that it has $m-1$ principle curvatures equal to

$$
\frac{y'(t)}{x(t)(x'(t)^2 + y'(t)^2)},
$$
$n-1$ principle curvatures equal to

$$
-\frac{x'(t)}{y(t)(x'(t)^2 + y'(t)^2)},
$$
and one principle curvature equal to

$$
\frac{x'(t) y''(t) - x''(t) y'(t)}{((x'(t)^2 + y'(t)^2)^{3/2}}
$$
Also, the unit normal vector (under projection $\Pi$) gives us the vector $\nu(t)$ perpendicular to $\gamma(t)$ as

$$
\nu(t) = \frac{(-y'(t), x'(t) )}{(x'(t) + y'(t))^{1/2}}
$$
(so calculated because the unit vector tangent to $\gamma(t)$ is $(x'(t), y'(t))$). Taken together, the $\lambda$-hypersurface equation reduces to 

\begin{align*}
\frac{1}{(x'(t)^2 + y'(t)^2)^{1/2}}\left( (m-1)\frac{y'(t)}{x(t)} - (n-1)\frac{x'(t)}{y(t)} + \frac{x'(t)y''(t) - x''(t) y'(t)}{x'(t)^2 + y'(t)^2} \right)& \\\\  = \frac{1}{2}\frac{x'(t)y(t) - x(t) y'(t)}{(x'(t)^2 + y'(t)^2)^{1/2}} &+ \lambda
\end{align*}
which can be rewritten as

\begin{align*}
\frac{x'(t)y''(t) - x''(t)y'(t)}{x'(t)^2 + y'(t)^2} =& \frac{1}{2}(x(t)y'(t) - x'(t) y(t)) + \frac{(n-1)x'(t)}{y(t)} \\ &- \frac{(m-1)y'(t)}{x(t)} + \lambda (x'(t)^2 + y'(t)^2)^{1/2}.
\end{align*}
If we introduce the angle $\theta(t) = \arctan\left( \frac{y'(t)}{x'(t)}\right)$, we get that 

$$
\theta'(t) = \frac{1}{1+(y'(t)^2 / x(t)'^2)} \left( \frac{x'(t) y''(t) - y'(t) x'')t)}{x'(t)^2}\right) = \left( \frac{x'(t)y''(t) - y'(t) x''(t)}{x'(t)^2 + y'(t)^2}\right).
$$

If we also assume that our profile curve is parametrized by arc length, we can use previous two formulas to show that the profile curve satisfies the following system of differential equations:

\begin{align} 
\dot{x} &= \cos \theta \label{2.1}\\
\dot{y} &= \sin \theta \label{2.2}\\
\dot{\theta} &= \left(\frac{x}{2} - \frac{m-1}{x} \right)\sin \theta + \left(\frac{n-1}{y} - \frac{y}{2} \right)\cos \theta + \lambda \label{2.3}
\end{align}

Similarly, it is clear that any curve in the first quadrant satisfying equations \eqref{2.1} - \eqref{2.3} will generate a hypersurface $\Sigma \subset \mathbb{R}^{n+m}$ that locally satisfies equation \eqref{1.4}.

\section{Analyzing the system of ODEs}

In this section, we record several explicit solutions to the system of equations \eqref{2.1}-\eqref{2.3} given above. We also perform some analysis on how solution curves can behave. We begin by identifying some explicit examples.

\begin{lemma}\label{lemma1}
We have the following explicit solutions to the system of ODES:
	\begin{enumerate}
	\item The horizontal line $y = \lambda + \sqrt{\lambda^2 + 2(n-1)}$ .
	\item The vertical line $x = \lambda + \sqrt{\lambda^2 + 2(m-1)}$.
	\item The circle of radius $\lambda + \sqrt{\lambda^2 + 2(m+n-1)}$.
	\end{enumerate}
\end{lemma}

\begin{proof}
These solutions can be verified by direct computation.
\end{proof}

\begin{lemma} \label{lemma32}
Let $\gamma(t)$ be a solution to the system of ODEs, and consider a subset of the solution $\gamma(t) = u(x)$ that is viewed as a graph over the $x$-axis. Then $u$ can only have maximums above a height of $y = \lambda + \sqrt{\lambda^2 + 2(n-1)}$, and can only have minimums below $y = \lambda + \sqrt{\lambda^2 + 2(n-1)}$. Similarly, if $\gamma(t) = v(y)$ is a graph over the $y$-axis, it can only have maximums (resp. minimums) at points below (resp. above) the line $x = \lambda + \sqrt{\lambda^2 + 2(m-1)}$.
\end{lemma}

\begin{proof}
This is seen by examining equation \eqref{2.3} at such a critical point. Note that we make use of $\lambda < 0$ in this argument.
\end{proof}

\begin{lemma} \label{lemma3}
Let $\gamma(t)$ be a solution to the system of ODEs \eqref{2.1} - \eqref{2.3}, defined on a time interval $t \in (a,b)$. If $x_{\gamma}(t) \rightarrow 0$ and $y_{\gamma}(t) \rightarrow y_b$, $y_b > 0$, as $t \rightarrow b$, then $\gamma$ can be extended to be defined on the interval $(a,b]$, such that $x_{\gamma}(b) = 0$, $y_{\gamma}(b) = y_b$, and $\theta_\gamma (b) = -\pi$. 
\end{lemma}

\begin{proof}
Let $\gamma(t)$ be a curve as described above. We first remark that, locally near $x=0$, $\gamma$ may be viewed as a function $u = u(x)$ over the $x$-axis. As seen above, such a function can only have maximums occur above the height $y = \lambda + \sqrt{\lambda^2 + 2(n-1)}$, while minimums can only occur below this height. Therefore, unless $y_b$ is exactly equal to this height, the function $u(x)$ will not exhibit oscillatory behavior as $x\rightarrow 0^+$. We'll consider the case where $y_b < \lambda + \sqrt{\lambda^2 + 2(n-1)}$, as the other case follows similarly. There are several possible behaviors as $t \rightarrow b$: either $\lim_{t \rightarrow b}\theta  < -\pi$, $\lim_{t \rightarrow b}\theta  = -\pi$, or $\lim_{t \rightarrow b}\theta  > -\pi$. Of course, in the second case our lemma is proven. Our goal is to show that the first and third case cannot occur.

To show that the third case cannot occur: by examining equation \eqref{2.3}, we see that

\begin{align*}
\dot{\theta} &= \left(\frac{x}{2} - \frac{m-1}{x} \right)\sin \theta + \left(\frac{n-1}{y} - \frac{y}{2} \right)\cos \theta + \lambda\\
&\geq   - \frac{m-1}{x} \dot{x} \tan \theta -  \left(\frac{n-1}{y_b} - \frac{y_b}{2} \right) + \lambda\\
&\geq   - \delta \frac{\dot{x}}{x} -  \left(\frac{n-1}{y_b} - \frac{y_b}{2} \right) + \lambda,
\end{align*}

where $\delta = (m-1) \liminf_{t \rightarrow b} \tan \theta$. Integrating this inequality from $t_1$ to $t_2$ gives us

\begin{align*}
\theta (t_2) - \theta (t_1) \geq \delta \ln \left( \frac{x(t_1)}{x(t_2)}\right) - \left(\frac{n-1}{y_b} - \frac{y_b}{2} \right)(t_2 - t_1) + \lambda (t_2 - t_1).
\end{align*}

Note that the expression $\theta(t_2) - \theta(t_1)$ is bounded, while the right-hand side blows up as $t_2 \rightarrow b$ - a contradiction.

As for the first case: if $\theta(t) < \-pi$ as $t \rightarrow b$, then for $t$ close to $b$ we can compute

\begin{align*}
\dot{\theta} &= \left(\frac{x}{2} - \frac{m-1}{x} \right)\sin \theta + \left(\frac{n-1}{y} - \frac{y}{2} \right)\cos \theta + \lambda\\
&\leq \frac{x}{2} - (m-1)\left( \frac{\dot{x}}{x}\tan \theta \right)\\
&\leq \frac{1}{2} + (m-1)\left( \delta \frac{\dot{x}}{x}\right),
\end{align*}
where $\delta = - \liminf_{t\rightarrow b} \tan \theta$ is a positive quantity. Integrating this inequality from $t_1$ to $t_2$ gives us

\begin{align*}
\theta (t_2) - \theta (t_1) < (m-1)\, \delta \, \ln \left( \frac{x(t_2)}{x(t_1)}\right) + \frac{1}{2}(x(t_2) - x(t_1)).
\end{align*}
Again, note that the expression $\theta(t_2) - \theta(t_1)$ is bounded, while the right-hand side goes to negative infinity as $t_2 \rightarrow b$ - a contradiction. Thus, the only situation that can occur is if $\gamma(t)$ meets the $y$-axis at exactly a perpendicular angle.

\end{proof}

A similar argument gives us an identical result for curves that intersect the $y$-axis, and gives us the following statement:

\begin{corollary}
If $\gamma(t)$ is a solution to our system \eqref{2.1} - \eqref{2.3} and intersects the $x$- or $y$-axis, it does so at a perpendicular angle.
\end{corollary}

Next, we see that straight lines are rare solutions to this sytem:

\begin{lemma}
The only straight lines that satisfy the system \eqref{2.1} - \eqref{2.3} are the two mentioned in Lemma \ref{lemma1}. In particular, there is no straight line through the origin that satisfies this system.
\end{lemma}

\begin{proof}
We can quickly see that the only vertical or horizontal lines that satisfy the system are the two already mentioned. A straight (non-vertical) line can be written in form $y = kx + b$ for some constants $k,b$. Using the fact that $$\dot{\theta} = \left(\frac{x}{2} - \frac{m-1}{x} \right)\sin \theta + \left(\frac{n-1}{y} - \frac{y}{2} \right)\cos \theta + \lambda,$$ as well as the fact that a linear solution would satisfy 

$$
\cos \theta = \frac{1}{\sqrt{1+k^2}}\,, \quad \sin \theta = \frac{k}{\sqrt{1 + k^2}}\,, \quad \text{and} \quad \dot{\theta} = 0\,,
$$
we get

$$
-\lambda \sqrt{1 + k^2} = \left( \frac{n-1}{kx + b} - \frac{kx + b}{2}\right) + k\left( \frac{x}{2} - \frac{m-1}{x} \right)
$$

Note that this can be rewritten as

$$
\frac{b}{2} - \lambda \sqrt{1 + k^2} = \left( \frac{n-1}{kx + b} \right) - k\left( \frac{m-1}{x} \right)
$$
Clearly, the left-hand side of this equation is constant, but no choice of $k$ and $b$ will fix the right-hand side of this equation, save $k=0$ - a case we have already examined. Therefore, no new linear solutions exist.
\end{proof}

\begin{remark} Since we already know all curves, including lines, must intersect the $x$- or $y$-axis away from the origin at a perpendicular angle, the main result in the previous lemma is that no line that passes through the origin is a solution. In \cite{McGrath15}, the diagonal line $L$ given by $y = \sqrt{\frac{m-1}{n-1}}x$ was a special solution. We do not have this solution in this situation. However, in the case where $m=n$, we still have symmetry of solutions over the line - a fact we will use in what follows.
\end{remark}

Finally, we conclude with an important lemma concerning the direction in which $\gamma(t)$ can ``curl,'' depending on where $\gamma(t)$ is located in regards to the line $L$, given by $y = \sqrt{\frac{m-1}{n-1}}x$.

\begin{lemma} \label{lemma5}
Suppose $\gamma(t) = (x(t),y(t))$ satisfies $(n-1)y^2 < (m-1)x^2$, so that $\gamma(t)$ lies below $L$. Furthermore, suppose that there is a time $t_0$ for which $\gamma$ satisfies $\dot{\theta}(t_0) < 0$, $\dot{x}(t_0) < 0$, and $\dot{y}(t_0) < 0$. Then for any $t$ in the maximal interval containing $t_0$ for which $\dot{x}(t) < 0$, $\dot{y}(t) < 0$, and $(n-1)y^2 < (m-1)x^2$, we will have $\dot{\theta}(t) < 0$.
\end{lemma}

\begin{proof}
We calculate
\begin{align}
\ddot{\theta} &=\dot{x}\dot{y} \left(\frac{m-1}{x^2} - \frac{n-1}{y^2} \right) + \dot{\theta} \left( \frac{x^2 - 2(m-1)}{2x} \cos \theta + \frac{y^2 - 2(m-1)}{2y} \sin \theta\right).
\end{align}
Note that when $\dot{\theta} = 0$, $\ddot{\theta} = \dot{x}\dot{y} \left(\frac{m-1}{x^2} - \frac{n-1}{y^2} \right)$. Since we are below the line $(n-1)y^2 < (m-1)x^2$, we get that when $\dot{\theta} = 0$, then $\ddot{\theta} <0$. This implies that $\theta$ will remain decreasing, with $\dot{\theta} < 0$.
\end{proof}

\section{Existence of a closed solution}

Our main argument will be to construct a closed curve that is a solution to the system \eqref{2.1} - \eqref{2.3}, specifically in the case where $n=m$. To this end, we construct a curve that satisfies the system of ODEs, and that lies entirely below the line $L$ (which, in this circumstance, is simply the line $y=x$), with both starting and ending point meeting $L$ perpendicularly. Because of the symmetry of our system of ODEs, we can then reflect our curve across the line $L$, ending with a closed loop that satisfies the ODE. 

Under the change of variables 
\begin{align}
r &= \frac{1}{\sqrt{2}}(x + y)\\
s &= \frac{1}{\sqrt{2}}(x-y)\\
\phi &= \arctan (s/r) = \pi/4 + \theta,
\end{align}
the system of ODEs becomes 

\begin{align}
\dot{r} &= \sin \phi\label{4.4}\\
\dot{s} &= \cos \phi\label{4.5}\\
\dot{\phi} &= \left(\frac{-r}{2} + \frac{(n-1)2r}{r^2-s^2} \right)\cos \phi + \left(\frac{s}{2} + \frac{(n-1)2s}{r^2-s^2} \right)\sin \phi + \lambda.\label{4.6}
\end{align}

Let $\gamma_{R}(t)$ denote a solution to \eqref{4.4} - \eqref{4.6} with initial conditions $r(0) = R$, $s(0) = \phi(0) = 0$, defined on some maximal time interval $[0, T_R)$. Note that, for for large starting $R$, $\dot{\phi}(0) \leq 0$, so $\gamma_R$ initially curls clockwise. Therefore, at least for a small amount of time, one can realize $\gamma_R$ as a positive, differentiable function over the $r$ axis. $T_R$ is taken to be the maximal time for which this remains a positive function - ie, $T_R$ is the first time at which either $s = 0$, $\phi = 0$, or $\phi = -\pi$. Under this identification, the function can be defined as as $s = f_R (r)$, for $r \in (r(\gamma_R(T_R)), R]$.

A major goal in this section will be to show that, for large enough $R$, $T_R$ will occur at a moment where $s = 0$.


\begin{lemma}
If $f_R$ has a critical point, then it is a maximum.
\end{lemma}

\begin{proof}
At such critical point, we would necessarily have $\phi = -\pi/2$. Then, at such a point, $\dot{\phi} =  -\left(\frac{s}{2} + \frac{(n-1)2s}{r^2-s^2} \right) + \lambda$. Since $\lambda$ is assumed to be negative and our curve is below $L$ (making $r^2 - s^2 > 0$), we have $\dot{\phi} < 0$ . Therefore, our critical point must have been a maximum.
\end{proof}

An immediate corollary to this is that $f_R$ has at most one critical point.

\begin{lemma}\label{lemma42}
For large values of $R$, $f_R$ will achieve a critical point (ie, a point where $\phi_R = -\pi/2$).
\end{lemma}

\begin{proof}
We adopt an argument from \cite{McGrath15} and rescale time by letting $\tau = R t$. Note that this means $\frac{d\tau}{dt} = R$, and therefore
\begin{align}
\frac{d\phi}{d\tau} &= \frac{r}{R}\cos \phi \left( -\frac{1}{2} + \frac{2(n-1)}{r^2-s^2}\right) + \frac{\lambda}{R} + \frac{1}{R} \sin \phi \left( \frac{s}{2} + \frac{2s(n-1)}{r^2-s^2}\right).\label{4.7}
\end{align}

First, we show that $\forall \epsilon$ and $\forall C$, there exists a $R$ such that $\forall \tau \in (0,C)$, we have $\frac{r(\tau)}{R} > (1-\epsilon)$. Indeed, this will be true since our (rescaled) path moves at a speed of $1/R$. Thus, by choosing $R$ large enough, we can guarantee that $r(\tau)$ stays close to $R$ in the time interval $(0,C)$. We make sure that $R \geq C$ and that $R$ is large enough to satisfy $\frac{R}{R-1} > (1-\epsilon)$, proving the claim.

This shows us that, for this large $R$, and for $\tau \in (0,C)$, we have

\begin{align}
\frac{d\phi}{d\tau} &\leq (1-\epsilon)\cos \phi \left( -\frac{1}{2} + \frac{2(n-1)}{r^2-s^2}\right) + \frac{\lambda}{R} + \frac{1}{R} \sin \phi \left( \frac{s}{2} + \frac{2s(n-1)}{r^2-s^2}\right) \label{48}\\
&\leq -\frac{1}{2}(1-\epsilon) \cos \phi.
\end{align}
To see this last inequality, note that the last term in \eqref{48} is negative (since $\phi < 0$), and the positive term $\frac{2(n-1)}{r^2 - s^2}$ is much smaller in size than the negative term $\frac{\lambda}{R}$ (since, if $R$ is large enough, $r$ will be close to $R$ and the squared term in the denominator will dominate). The equation $d\phi / d\tau = -\frac{1}{2}(1-\epsilon) \cos \phi$ has the explicit solution of 
\begin{align}
\phi (\tau) &= -2 \arctan \left( \tanh \left( \frac{(1-\epsilon)\;\tau}{4} \right) \right)
\end{align}
which will govern the behavior of $\phi$ for large initial $R$. Note that
\begin{align}\label{411}
-2\arctan \left( \tanh \left( \frac{(1-\epsilon)\tau}{4}\right)\right) + \frac{\pi}{2} &= \mathcal{O}(e^{-\frac{(1-\epsilon)\tau}{2}}),
\end{align}
which implies that (again, for large initial $R$), our curve will initially curl clockwise and have $\phi \approx -\pi/2$ exponentially quickly in the $\tau$-parameter. Furthermore, \eqref{411} implies that, for a small fixed number $\tau_0$, there exists a constant $c$ such that, as long as $\frac{d\phi}{d\tau} < 0$ and $\tau > \tau_0$, we have $s(\tau) > \frac{c}{R}$. This, combined with \eqref{4.7}, implies that there is a $\tau$ for which $\phi (\tau) = -\pi / 2$. Furthermore, \eqref{411} implies that, for large $R$,  
$r_0 = R - \mathcal{O}(1/R)$ and $s_0 = \mathcal{O}(1/R)$.

\end{proof}

Combining this with Lemma \ref{lemma5} gives us the following corollary:
\begin{corollary}\label{corollary43}
For large $R$, $\phi_R$ is decreasing until either $\gamma$ crosses the line $L$, or $\phi = -3\pi / 4$.
\end{corollary}

\begin{proof}
The previous lemma tells us that $\phi$ will decrease to $-\pi/2$. At this point, $\dot{x}<0$ and $\dot{y}<0$, so we can apply Lemma \ref{lemma5} to complete the argument.
\end{proof}

\begin{lemma}\label{lemma8}
For large $R$, $f_R (T_R) = 0$.
\end{lemma}

\begin{proof}

Corollary \ref{corollary43} tell us that that, for large $R$, we have $\dot{\phi}_R < 0$ at least until $\phi < -\frac{3\pi}{4}$ or $f_R (T_R) = 0$. At the same time, Lemma \ref{lemma42} tells us that, for large $R$, the maximum of $f_R$ is $\mathcal{O}(\frac{1}{R})$. Furthermore, we can compute that

\begin{align*}
\frac{d\phi}{dr} &= \dot{\phi} / \dot{r} \\ 
&= \left(\frac{-r}{2} + \frac{(n-1)2r}{r^2-s^2} \right)\cot \phi + \left(\frac{s}{2} + \frac{(n-1)2s}{r^2-s^2} \right) + \lambda \csc \phi \\
&= I \cot \phi + II + \lambda \csc \phi
\end{align*}

If $\phi$ were to equal $-3\pi / 4$, this equation simplifies to 
\begin{align}\label{4.12}
\frac{d\phi}{dr} = I + II - \lambda \sqrt{2}.
\end{align}

Note that when $r$ is large and $s = \mathcal{O} (1/R)$ is small, we know that $II$ is a small, positive quantity, while $I$ is large and negative. Therefore, there exists a $C_1$ and $C_2$ (that depend only on $\lambda$ and $n$) such that, if $r > C_1$ and $R > C_2$ , $\frac{d\phi}{dr}$ will be negative at any point where $\phi = -3\pi / 4$. This, would make $\dot{\phi}$ positive. Therefore, for $R > C_2$, $\phi$ cannot be less than $-3\pi/4$ for $r \in [C_1, R]$.







Let's assume that we can find increasingly large $R$ for which $f_R$ remains positive on the interval $[C_1, R]$. The work above shows that, on this interval, the function $f_R$ satisfies $\dot{\phi} < 0$. Therefore, on a slightly smaller sub-interval (for which $\phi < -\pi / 2$), we have $\frac{d\phi}{dr} > 0$. Notice that the region (over $r$) where $\frac{d\phi}{dr} > 0$ corresponds to a region when $f_R$ is concave down. This gives us a way to estimate how negative $\phi$ can be. Indeed, we see that at the point $r = 2C_1$, the angle $\phi$ must satisfy
\begin{align*}
\phi > -\frac{\pi}{2} - \arctan \left( \frac{1}{C_1 R} \right),
\end{align*}
or else $f$ would cross the $r$-axis somewhere on the interval $[C_1, 2\,C_1]$. Therefore, for the point $r_0 = 2C_1$, we have that $f_R(r_0) = \mathcal{O}(R^{-1})$ and $\phi_R(r_0) + \pi / 2 = \mathcal{O}(R^{-1})$. 

From this, and the smooth dependence of ODEs on their initial conditions, it is evident that the solutions $f_R$ (again, for increasingly large $R$) converge to a solution of the original system of ODEs that passes through the point $r = 2C_1$, $s=0$, $\phi = -\pi/2$. However, such a solution has $\dot{\phi} = \lambda < 0 $, so the solution instantly moves above the line $L$. This implies that, for $R$ large enough, our solution will have a point $r \in [C_1 , 2C_1]$ that satisfies $f_R (r) =0$ and $\phi > -\pi$, giving us a contradiction.

\end{proof}

By Lemma \ref{lemma8}, $\gamma_R(T_R)$ will occur when $\gamma_R$ intersects the line $L$ for all $R$ that are large enough. However, we know that there exists an explicit solution for which $\gamma_R (T_R)$ ends on the $x$-axis (this is the circle solution, described in \ref{lemma1}). This implies that the point $R_* := \inf\{ R>0 : f_{\bar{R}}(r_{\bar{R}} (T_R)) = 0\; \text{for all}\; \bar{R} > R \}$ exists, is well-defined, and is greater than 0. Our next goal is to show that, as $R \searrow R_*$, the solutions $\gamma_R$ stay away from the $y$ axis and the origin.

\begin{lemma} \label{lemma9}
$R_*$ satisfies $\liminf\limits_{R \searrow R_*} \left( \min\limits_{t < T_R} y_R(t)\right) > 0$. 
\end{lemma}

\begin{proof}
Suppose this were not true, ie suppose there was a sequence of points $R_m \searrow R_*$ and a sequence of times $t_m$ such that $y_{R_{m}}(t_m) \rightarrow 0$. Passing to a subsequence if necessary, we can assume that $x_{R_m}(t_m) \rightarrow x_*$, so that the curves $\gamma_{R_m}$ converge to a curve $\gamma_*$ with $\gamma_{R_m} (t_m) \rightarrow (x_*, 0)$. First, assume that $x_* > 0$. Then by Lemma \ref{lemma3}, $\gamma_*$ will intersect the $x$-axis orthogonally. Because of the continuity of solutions, and the fact that all $\gamma_R$ for $R > R_*$ end on the line $L$, we know that for $R$ just above $R_*$, $\gamma_R$ travels towards the point $(x_*, 0)$ in a manner perpendicular to the $x$-axis, almost touches the $x$-axis, and then rapidly curls around and moves away from the point in a nearly vertical way. In particular, for $R_m$ very close to $R_*$, the curve $\gamma_{R_m}$ will no longer be a graph over the line $L$. This implies that at this particular $R_M$, $T_{R_M}$ occurs before $\gamma_{R_M}$ returns to $L$, which is a contradiction.

Next, assume that $x_* = 0$, so our $\gamma_{R_m}(t_m)$ are converging to the origin. This implies that the solution $\gamma_{R_*}$ terminates at $\gamma_{R_*}(T_{R_*}) = (0,0)$, and that $\gamma$ stays below the line $L$ close to the origin. Because of Lemma \ref{lemma5}, we know that $\theta$ will continue to decrease, and therefore that as $\gamma_{R_*}(T_{R_*})$ intersects $(0,0)$, the angle of approach is $\lim_{t\rightarrow T_{R_*}} \theta(t)$ = $\alpha$ for some angle $\alpha < -3\pi / 4$. 

For this curve $\gamma_{R_*}$, and for $t$ very close to $T_{R_*}$, we have
\begin{align*}
\dot{x}, \dot{y} &< 0 \\
y &< x \tan \alpha \\
\lambda &< \frac{y}{2} \cos \theta
\end{align*}

Using these ingredients, we compute that

\begin{align}
\dot{\theta} &= \left(\frac{x}{2} - \frac{n-1}{x} \right)\sin \theta + \left(\frac{n-1}{y} - \frac{y}{2} \right)\cos \theta + \lambda \\
&\leq \left(n-1\right)\left(1 - \tan \alpha \right) \frac{\dot{y}}{y}
\end{align}
which is a negative quantity since $\tan \alpha < 1$. Integrating from $t_1$ to $t_2$ gives us

\begin{align}
\theta (t_2) - \theta (t_1) \leq \left(n-1 \right)\left(1 - \tan \alpha \right) \log \left( \frac{y(t_2)}{y(t_1)}\right).
\end{align}

We know that $\theta(t_2) - \theta(t_1)$ is negative and bounded, which means that $\log\left(\frac{y(t_2)}{y(t_1)}\right)$ cannot grow to $-\infty$. However, as $t_2 \rightarrow T$, we have $y(t_2) \searrow 0$, which is a contradiction.

\end{proof}

\begin{proof}[Proof of Theorem 1]
Because of Lemma \ref{lemma9}, we know that the solutions as $R \searrow R_*$ stay in a compact set away from the $x$ axis and the origin. Because of the continuity of the system of ODEs, this guarantees that the solution originating from $R_*$ begins and ends on the line $L$. We will now show that $\gamma_{R_*}$ must end by intersecting $L$ perpendicularly, at an angle of $\phi = -\pi$.

If $\phi_{R_*} (t_M) > -\pi$, then (by the continuity of the system of ODEs) there would exist a $R_{\circ} < R_*$ for which, for all $R$ satisfying $R_{\circ} < R \leq R_{*}$, $\gamma_{R}$ is a graph over the line, begins and ends on the line, and ends also at an angle $\phi_R (t_M) > -\pi$. This contradicts the definition of $R_*$ as the infimum of all such values.

At the same time, if $\gamma_{R_*}$ meets the line at an angle $\phi_{R_*} < -\pi$, then (again by the continuity of the system of ODEs) there would exist a value $R_{\circ} > R_*$ that also satisfies $\phi_R < -\pi$. This would imply that  contradict the notion that $R_{*}$ was the infimum, as the infimum must be greater than or equal to $R_{\circ}$.

Therefore, $\gamma_{R_*}$ must begin and end on the line, and must meet the line perpedicularly. This completes the theorem.

\end{proof}

\bibliographystyle{plain}
\bibliography{Summer2017}

\end{document}